\documentclass[10pt,a4paper]{amsart}
\usepackage[latin1]{inputenc}
\usepackage{amsmath,amsfonts,amssymb,amsthm,bookmark,appendix,graphicx}
\usepackage{mathrsfs} %to type beautiful Y
\usepackage{dsfont} %to type characteristics 1
\usepackage{tikz}
\usepackage{caption}
\usepackage{subcaption}

\newcommand\copyrighttext{%
  \footnotesize 
  This is an author-created, un-copyedited version of an article accepted for publication in ESAIM: COCV. The publisher is not responsible for any errors or omissions in this version of the manuscript or any version derived from it.}
\newcommand\copyrightnotice{%
\begin{tikzpicture}[remember picture,overlay]
\node[anchor=south,yshift=740pt] at (current page.south) {\fbox{\parbox{\dimexpr\textwidth-\fboxsep-\fboxrule\relax}{\copyrighttext}}};
\end{tikzpicture}%
}

\theoremstyle{plain}
\newtheorem{theorem}{Theorem}%[section]
\newtheorem{lemma}[theorem]{Lemma}

\theoremstyle{remark}
\newtheorem*{definition}{Definition}
\newtheorem*{remark}{Remark}
\newtheorem{example}[theorem]{Example}

\def\bq{\begin{eqnarray}}
\def\eq{\end{eqnarray}}
\def\bqq{\begin{eqnarray*}}
\def\eqq{\end{eqnarray*}}
\def\nn{\nonumber}
\def\minus {\backslash}
\def\eps{\varepsilon}

\def\fall{~\text{for all}~}

\def\d {\partial}
\def\E {\mathscr{E}}

\def\real {\mathds{R}}
\def\N {\mathds{N}}
\def\diss {\mathscr{D}iss}

\def\nn {\nonumber}

\def\N0 {\mathop N\limits^{\circ}}

\def\fall{~\text{for all}~}
\def\diss {\mathscr{D}iss}
\def\Diss {\mathscr{D}iss}

\title[BV solutions via epsilon-neighborhood method]{BV solutions constructed using the epsilon-neighborhood method}
\author[Mach Nguyet Minh]{}

 \email{machnt@mathematik.uni-stuttgart.de}
\begin{document}

\copyrightnotice

\date{{June 20, 2012}}
\maketitle

% Enter the first author's name and address:
\centerline{\scshape Mach Nguyet Minh}
\medskip
{\footnotesize
% please put the address of the first author
 \centerline{Department of Mathematics}
   \centerline{University of Stuttgart}
   \centerline{ Allmandring 5b, Stuttgart, 70569, Germany}
} % Do not forget to end the {\footnotesize by the sign }

\begin{abstract} We study a certain class of weak solutions to rate-independent systems, which is constructed by using the local minimality in a small neighborhood of order $\eps$ and then taking the limit $\eps\to 0$. We show that the resulting solution satisfies both the weak local stability and the new energy-dissipation balance, similarly to the BV solutions constructed by vanishing viscosity introduced recently by Mielke, Rossi and Savar\'e \cite{MieRosSav-10,MieRosSav-12,MieRosSav-13}.
\end{abstract}

%%%%%%%%%%%%%%%%%%%%%%%%%%%%%%%%%%%%%%%%%%%%%%%%%%%%%%%
%%%%%%%%%%%%%%%%%%%%%%%%%%%%%%%%%%%%%%%%%%%%%%%%%%%%%%%%%%%%%
\section{Introduction}
A rate-independent system is a specific case of quasi-static systems. It is time-dependent but its behavior is slow enough so that the inertial effects can be ignored and the systems are only affected by external loadings. Some specific  rate-independent systems have been studied by many authors, including Francfort, Marigo, Larsen, Dal Maso and Lazzaroni on brittle fractures \cite{FraMar-98, FraLar-03, Larsen-10, DalLaz-10}, Dal Maso, DeSimone and Solombrino on the Cam-Clay model \cite{DalDesSol-08}, Dal Maso, DeSimone, Mora, Morini on plasticity with softening \cite{DalSesMorMor-08a, DalDesMorMor-08b}, Mielke on elasto-plasticity \cite{Mielke-02, Mielke-03}, Mielke, Theil and Levitas on shape-memory alloys \cite{MieThe-99, MieThe-04, MiThLe-02}, M\"{u}ller, Schmid and Mielke on super-conductivity \cite{Muller-99, SchMie-05}, and Alberti and DeSimone on capillary drops \cite{AlbDeS-11}. We refer to the surveys \cite{Mielke-06, Mielke-05, Mielke-07, Mielke-08} by Mielke for the study in abstract setting as well as for further references.

In this work, we consider a finite-dimensional normed vector space $X$, an evolution $u: [0,T]\to X$ subject to a force defined by an energy functional $\E: [0,T] \times X \to [0,+\infty)$ which is of class $C^1$, and a dissipation function $\Psi(x)$ which is convex, non-degenerate and positively $1$-homogeneous. Given an initial position $x_0\in X$ which is a local minimizer for the functional $x\mapsto \E(0,x)+\Psi(x-x_0)$, we say that $u$ is a solution to the rate-independent system $(\E,\Psi,x_0)$ if $u(0)=x_0$ and the following inclusion holds true,
\bq \label{eqn:Evol-eqn1a}
0 \in \partial \Psi(\dot{u}(t)) + D_x \E(t,u(t)) \text{ in } X^*,  \text{ for a.e. } t \in (0,T),
\eq
where $X^*$ denotes the dual space of $X$, $\partial \Psi$ is the subdifferential of $\Psi$ and $D_x\E$ is the differential of $\E$ w.r.t. the spatial variable $x$.

In general, strong solutions to (\ref{eqn:Evol-eqn1a}) may not exist \cite{Stefanelli-09}. Hence, the question on defining some {\it weak solutions} arises naturally. 

A widely-used weak solution is the {\it energetic solution}, which was first introduced by Mielke and Theil \cite{MieThe-99} (see \cite{MieThe-04, MaiMie-05, FraMie-06, Mielke-06} for further studies). A function $u: [0,T] \to X$ is called an energetic solution to the rate-independent system $(\E,\Psi,x_0)$ if it satisfies: 
\begin{itemize}
\item [(i)] the initial condition $u(0)=x_0$;

\item[(ii)] the \textit{global stability} that for  $(t,x) \in [0,T]\times X$,
\bq \label{eq:global-stability}
\E(t,u(t)) \le \E(t,x)+\Psi(x-u(t));
\eq

\item[(iii)] the \textit{energy-dissipation balance} that for all $0 \le t_1 \le t_2 \le T$,
\bq \label{eq:energy-dissipation-old}
\E(t_2,u(t_2))-\E(t_1,u(t_1))=\int_{t_1}^{t_2} {\partial_t \E(s,u(s))\,ds} - \diss_{\Psi}( u;[t_1,t_2]) .
\eq
\end{itemize}

Here, $\diss_{\Psi}$ is the usual total variation induced by $\Psi(\cdot)$   
\bqq
\diss_{\Psi}( u(t);[t_1,t_2]):=\sup\left\{ \sum_{i=1}^N \Psi(u(s_{i})-u(s_{i-1})) ~\huge{|}~N \in \mathbb{N}, \; t_1= s_0<s_1<\dots<s_N= t_2 \right\}.
\eqq

Note that when the energy functional is not convex, the global minimality (\ref{eq:global-stability}) makes the energetic solutions jump sooner than they should, and hence fails to describe the related physical phenomena (see Examples \ref{example2} below). Therefore, some weak solutions based on {\it local minimality} are of interest. 

Recently, an elegant weak solution based on vanishing viscosity method has been introduced by Mielke, Rossi and Savar\'e \cite{MieRosSav-10, MieRosSav-12, MieRosSav-13}. Their idea is to add a small viscosity term to the dissipation functional $\Psi$. This results in a new dissipation functional $\Psi_{\eps}$, e.g. $\Psi_{\eps}(x) = \Psi(x) + \frac{\eps}{2}\|x\|^2$, which has super-linear growth at infinity and converges to $\Psi$ in an appropriate sense as $\eps$ tends to zero. They showed that the modified system $(\E, \Psi_{\eps},x_0)$ admits a solution $u_\eps$. The limit $u$ of a subsequence $u_\eps$ as $\eps\to 0$, called {\it BV solution}, has the following properties:
\begin{itemize}

\item[(i)] the initial condition $u(0)=x_0$;

\item[(ii)] the {\it weak local stability} that for all $t\in [0,T]\minus J$, 
\bq \label{eq:weak-local-stability}
-D_x \E(t,u(t)) \in  \partial \Psi(0);
\eq

\item[(iii)] the {\it new energy-dissipation balance} that for all $0 \le t_1 \le t_2 \le T$,
\bq \label{eq:energy-dissipation-new}
\E(t_2,u(t_2))-\E(t_1,u(t_1))=\int_{t_1}^{t_2} {\partial_t \E(s,u(s))\,ds} - \diss_{new}(u;[t_1,t_2]) .
\eq
 
\end{itemize}
Here, $J$ is the {\it jump set} of $u$ on $[0,T]$
\[J:=\{t\in [0,T]~|~u(\cdot) {\rm~ is~ not~ continuous ~at~}t\},\]
$\partial \Psi(0)$ is the subdifferential of $\Psi$ at $0$,  $\left<\cdot,\cdot\right>$ is the dual pairing between $X^*$ and $X$
\[\partial \Psi(0):= \{ \eta \in X^*\;|\; \langle \eta,v \rangle \le \Psi(v) \; \forall v \in X\},\]
and the {\it new dissipation} is defined by
\bqq
\diss_{ new}(u;[t_1,t_2])&:=&\diss_{\Psi}(u;[t_1,t_2])+\sum_{t \in J \cap (t_1,t_2)} \left(  \Delta_{new}(t;u(t^-),u(t))+\Delta_{new}(t;u(t),u(t^+))  \right) \hfill\\
&& + \Delta_{new}(t_1;u(t_1),u(t_1^+)) + \Delta_{new}(t_2;u(t_2^-),u(t_2))\hfill\\
&& - \sum_{t \in J \cap (t_1,t_2)} \left( \Psi(u(t)-u(t^-))+\Psi(u(t^+)-u(t))\right) - \Psi(u(t_1^+)-u(t_1)) -\Psi(u(t_2)-u(t_2^-)),
\eqq
where $\Delta_{new}(t;a,b)$ also depends on the energy functional $\E$, the dissipation $\Psi$ and the viscous norm $\|\cdot\|$
\bqq
\Delta_{new}(t;a,b)=\inf \left\{\int_0^1 \left(\Psi(\dot{\gamma}(s))+\|\dot{\gamma}(s)\|\cdot \inf_{z \in \partial \Psi(0)}\|D_x \E(t,\gamma(s))+z\|_{*}\right)ds\;|\; \gamma \in AC([0,1];X), \gamma(0)=a, \gamma(1)=b \right\}.
\eqq
Here, the dual norm of $\|\cdot\|$ is defined by $\|\eta\|_*:=\sup_{v \in X\minus \{0\}}\frac{| \left<\eta,v\right> |}{\|v\|}$ for all $\eta\in X^*$.

The new energy-dissipation balance is a deep insight observation, which contains the information at the jump points. Indeed, it has been shown in \cite{MieRosSav-12, MieRosSav-13} that if the BV solution $u$ jumps at time $t$, there exists an absolutely continuous path $\gamma:[0,1] \to X$, which called an {\it optimal transition} between $u(t^-)$ and $u(t^+)$, such that
\begin{itemize}

\item[(i)] $\gamma(0)=u(t^-)$, $\gamma(1)=u(t^+)$, and there exists $s\in[0,1]$ such that $\gamma(s)=u(t)$;

\item[(ii)] for all $s \in [0,1], -D_x \E(t,\gamma(s))$ stays outside the set $\partial \Psi(0)$ (if $\gamma$ is of viscous type), or on the boundary of $\partial \Psi(0)$ (if $\gamma$ is of sliding type);

\item[(iii)] $\E(t,u(t^{-}))-\E(t,u(t^{+}))=\int_0^1 \left( \Psi(\dot{\gamma}(s)) +\|\dot{\gamma}(s)\|\cdot \inf_{z \in \partial \Psi(0)} \| D_x\E(t,\gamma(s)) +z\|_{*}\right)\,ds$.
\end{itemize}

%{\bf Should omit the following sentence!!!! Besides, should add 2 more cases in Example 1.2, one for correct BV-viscosity, one for wrong BV-neighbor ---> Should be done when prepare Report 2.}

As we can see from the definition, BV solutions constructed using the vanishing viscosity method also depend on viscosity. Usually, viscosity arises naturally from physical models.

To deal with local minimizers but with a totally different approach, Larsen \cite{Larsen-10} proposed the $\eps$-stability solution in the context of fracture mechanics. The idea is to choose minimizers among all $\eps$-accessible states w.r.t. the discretized solution at previous time-step. A state $v$ is called $\eps$-accessible w.r.t. state $z$ if the total energy at $v$ is lower than the total energy at $z$; and there is a continuous path connecting $z$ to $v$ such that total energy never increases by more than $\eps$ along this path. In this way, the limit $u(t)$, when passing from discrete to continuous time, satisfies the $\eps$-stability: $u(t)$ is $\eps$-stable at every time $t$, i.e. there is no $\eps$-accessible state w.r.t. $u(t)$. A similar version of {\it optimal transition} is obtained at jump points: if the solution jumps at time $t$, there exists a continuous path connecting $u(t^-)$ to $u(t^+)$ such that total energy increases no more than $\eps$ along this path. The energy-dissipation upper bound is proved for fixed $\eps>0$. The energy-dissipation equality is obtained if the solution has only jumps of sizes less than $\eps$.

%However, the BV solution constructed by vanishing viscosity depends heavily on the choice of the viscosity, and the solution obtained by some viscosity does not have expected behavior (see Examples \ref{example2}). {\bf Stop!}

In this work, we shall discuss one more way to deal with local minimizers. The idea is similar to the viscosity method of Mielke-Rossi-Savar\'e in \cite{MieRosSav-10, MieRosSav-12, MieRosSav-13}, but instead of adding a small viscosity into the dissipation, we consider the minimization problem (\ref{eq:global-stability}) in a small neighborhood of order $\eps$. Passing from discrete to continuous time, we obtain a limit $x^{\eps}(\cdot)$. Then, by taking $\eps\to 0$, we get a solution $u(\cdot)$. The epsilon-neighborhood approach was first suggested in \cite[Section 6]{Mielke-03} for the one-dimensional case when $\eps$ is chosen proportionally to the square root of the time-step and the weak local stability was then obtained in \cite{EfeMie-06}. 
%Notice that the time-step in discrete scheme in epsilon-neighborhood method is fixed, while the step-size in \cite{EfeMie-06} is modified after each step. 
%In this work, we shall study another construction of BV solutions, called BV solutions constructed by epsilon-neighborhood method. The idea is to consider the minimization problem (\ref{eq:global-stability}) in a small neighborhood of order $\eps$ to obtain a solution $x^{\eps}(\cdot)$, and then take $\eps\to 0$ to get a limit $u(\cdot)$,  which satisfies the definition of BV solution proposed by Mielke-Rossi-Savar\'e in \cite{MieRosSav-10, MieRosSav-12, MieRosSav-13}. The epsilon-neighborhood approach was first suggested in \cite[Section 6]{Mielke-03} for one-dimensional case when $\eps$ is chosen proportional to the square root of the time-step and the weak local stability was then obtained in \cite{EfeMie-06}. {\bf In Example 1.2 below, we give a comparison between the solutions constructed with the method in \cite{EfeMie-06} and those constructed with the epsilon-neighborhood method. (NOT YET DONE!)} 

%{\bf Add a remark here: 1. State clear that eps-neigh is not included in MRS; 2. Some comments about Paper 14 of Mielke (to answer remark 2, report 1.}

Roughly speaking, the epsilon-neighborhood method is a special case of the vanishing viscosity approach when the viscosity term is chosen as follows
\[ \Psi_0(v):= \begin{cases}0 & \text{ if } |v| \le 1, \\ +\infty & \text{ if } |v|>1. \end{cases} \]
However, this method was not discussed in \cite{MieRosSav-12, MieRosSav-13} since the viscosity there is required to be finite (see \cite[Section 2.3]{MieRosSav-12} and \cite[Section 2.1]{MieRosSav-13} for further discussions).
%$\Psi_0$ does not quite satisfy the requirement to become a viscosity in vanishing viscosity in \cite[Section 2.3]{MieRosSav-12}. {\bf Stop!}

In this article, we shall show that the BV solution constructed using the epsilon-neighborhood method $u(\cdot)$ satisfies both the weak local stability and the new energy-dissipation balance, i.e. it satisfies the definition of BV solutions introduced by Mielke, Rossi and Savar\'e \cite{MieRosSav-10, MieRosSav-12,MieRosSav-13}. Similar to BV solutions constructed by vanishing viscosity, BV solutions constructed using the epsilon-neighborhood method also depend on the norm that defines the ``neighborhood''. In Example 1.2 below, we shall make a comparison between different notions of weak solutions, i.e. energetic solutions, BV solutions constructed by vanishing viscosity, BV solutions constructed by the epsilon-neighborhood method as well as the solutions constructed by the method in \cite{EfeMie-06}. For a detailed discussion on the different notions of weak solutions, we refer to the papers \cite{Mielke-08, Negri-10, RoSa-12, Minh-12}.
%\tableofcontents

\text{}\\
{\bf Acknowledgments.} I am indebted to Professor Giovanni Alberti for proposing to me the problem and giving many helpful discussions. I warmly thank Professor Riccarda Rossi, Li-Chang Hung and Tran Minh-Binh for their helpful comments and remarks. I really appreciate the three referees for many enlightening and insightful remarks and helpful suggestions and corrections. This work has been partially supported by the PRIN 2008 grant ``Optimal mass transportation, Geometric and Functional Inequalities and Applications" and the FP7-REGPOT-2009-1 project ``Archimedes Center for Modeling, Analysis and Computation".

\section{Main results}
%{\bf Need to be revised: Precise conditions for norm and $\Psi$. Add a remark if the result holds for Hilbert space.} 
For simplicity, we shall consider the case when $X=\real^d$ and the unit ball of the norm $\|\cdot\|$ which defines the neighborhood has $C^1$-boundary. In addition, we assume that the energy functional $\E(t,x):[0,T]\times \real^d \to [0,\infty)$ satisfies the following technical assumption: there exists $\lambda = \lambda(\E)$ such that
\bq \label{eq:E1}
\left| \partial_t \E (s,x) \right| \le \lambda \,\E(s,x)  \text{ for all } (s,x) \in [0,T] \times \real^d. 
\eq
\begin{remark} The condition (\ref{eq:E1}) was proposed in \cite{Mielke-08}. Together with Gronwall's inequality, (\ref{eq:E1}) implies that
\begin{eqnarray} \label{E2}
\E(r,x) \le \E(s,x)\, e^{\lambda|r-s|}, ~ \left| \partial_t \E(r,x) \right| \le \lambda \,\E(s,x)\, e^{\lambda|r-s|}
\end{eqnarray}
for any $r, s$ in $[0,T]$.
\end{remark}

\begin{definition}[Construction of discretized solutions] Let $\eps>0$, $\tau>0$ and let $N\in \mathbb{N}$ satisfy $T\in [\tau N, \tau (N+1))$. We define a sequence $\{x^{\eps,\tau}\}_{i=0}^N$ by $x^{\eps,\tau}_0=x_0$ (initial position) and 
\bq\label{eq:dis-prob}
x_i^{\eps,\tau}\in {\rm argmin} \{\E(t_i,x)+\Psi(x-x_{i-1}^{\eps,\tau})\; | \; \|x-x^{\eps,\tau}_{i-1}\| \le \eps\}~~{\rm for ~every}~ i \in \{1, \dots, N\}.
\eq
The discretized solution $x^{\eps,\tau}(\cdot)$ is then constructed by interpolation  
$$x^{\eps,\tau} (t):= x^{\eps,\tau} _{i-1}~{\rm for ~every}~t \in [t_{i-1},t_i),i \in \{1, \dots, N\}.$$
\end{definition}

Our main result is as follows. 

\begin{theorem}[BV solutions constructed using the epsilon-neighborhood method]\label{thm:eps-neigh-sol} Let $\E: [0,T] \times \real^d \to [0,+\infty)$ be of class $C^1$ and satisfy (\ref{eq:E1}). The dissipation functional $\Psi: \mathbb{R}^d \to [0,\infty)$ is assumed to be convex, positively $1$-homogeneous and satisfy $\Psi(v)>0$ for all $v \in \mathbb{R}^d\minus\{0\}$. Given an initial datum $x_0\in \real^d$ which is a local minimizer of the functional $x\mapsto \E(0,x)+\Psi(x-x_0)$. Then, we have the following properties:

\begin{itemize}

\item[(i)] (Discretized solution) For any $\eps>0$ and $\tau>0$, there exists a discretized solution $t \mapsto x^{\eps,\tau}(\cdot)$ as described above.

\item[(ii)] (Epsilon-neighborhood solution) For any fixed $\eps>0$, there exists a subsequence $\tau_n \to 0$ such that $x^{\eps,\tau_n}(\cdot)$ converges pointwise to some limit $x^{\eps}(\cdot)$. Moreover,

\item (Epsilon local stability) If $x^{\eps}(\cdot)$ is  
right-continuous at $t$, namely $\lim_{t'\to t^+} x^{\eps}(t')=x^{\eps}(t)$, then $x^{\eps}(t)$ satisfies the epsilon local stability
\begin{equation*}
 \E(t,x^{\eps}(t)) \le  \E(t,x)+\Psi(x-x^{\eps}(t))~~{\rm for~all}~\|x-x^{\eps}(t)\|\le \eps;%\tag{eps-LS}
 \end{equation*}

\item (Energy-dissipation inequalities) We have $\diss_{\Psi} (x^{\eps};[0,T])\le C$ (independent of $\eps$), $\d_t \E(\cdot,x ^{\eps}(\cdot))\in L^1(0,T)$ and for all $ 0\le s \le t \le T$,
$$  -\diss_{new}(x^{\eps};[s,t]) \le \E(t,x^{\eps}(t))-\E (s,x^{\eps}(s)) - \int_{s}^t \d_t \E(r, x^{\eps}(r))\,dr\le  -\diss_{\Psi}(x^{\eps};[s,t]).$$

\item[(iii)] (BV solutions constructed by epsilon-neighborhood) There exists a subsequence $\eps_n\to 0$ such that $x^{\eps_n}$ converges pointwise to some BV function $u$. Furthermore, the function $u$ satisfies: 

\item (Weak local stability) If $t \mapsto u(t)$ is continuous at $t$, then
%\begin{eqnarray}\label{stability-1}
\[ -\nabla_x \E(t,u(t)) \in \partial \Psi(0);\]
% \end{eqnarray}

\item (New energy-dissipation balance) For all $0\le s\le t\le T$, one has
$$  \E(t,u(t))-\E (s,u(s))= \int_{s}^t \d_t \E(r, u(r))\,dr-\diss_{new}(u;[s,t]).$$
\end{itemize}

\end{theorem}

An explicit example is given below (a detailed explanation can be found in the Appendix).
\begin{example}\label{example2}
Consider the case $X=\real$, $\Psi(x)=|x|$, $x_0=0$ and the energy functional 
$$\E(t,x):=x^2-x^4+0.3\,x^6+t\,(1-x^2)-x+6,~~t\in [0,2].$$
\begin{itemize}
\item[(i)] The strong solution is $x(t)=0$ for $t \in [0,1)$. This solution cannot be extended continuously when $t \ge 1$, since it would violate the local minimality. 
\item[(ii)] The energetic solution constructed by time-discretization satisfies
\bqq
x(t)=0~\text{ if }~ t <\frac{1}{6}\,\,,\, x(1/6)\in \{0, \sqrt{5/3}\}\;~ \,\,\text{and}~x(t) = \frac{\sqrt{10+\sqrt{10+90t}}}{3} \text{ if }  t>\frac{1}{6}.
\eqq
This solution jumps at $t=1/6$, from $x=0$ to $x=\sqrt{5/3}$, but this jump is not physically relevant (see Figure \ref{fig:fig1} below). The energetic solution satisfies the energy-dissipation balance but it does not satisfy the new energy-dissipation balance. 
\item[(iii)] The BV solution corresponding to the viscous dissipation $\Psi_{\eps}(x)=|x|+\eps x^2$ is 
\[ x(t)=0 ~{\rm for~all}~t\in [0,2].\]
When $t \ge 1$, this solution violates the local minimality.
\item[(iv)] The BV solution constructed by the epsilon-neighborhood method is
\bqq
x(t)=0~~ \text{ if }~ t <1~\,\,\text{and}~ \; x(t) = \frac{\sqrt{10+\sqrt{10+90t}}}{3} \text{ if }  t>1.
\eqq
This solution coincides with the strong solution up to the point where the strong solution exists. Moreover, it jumps at $t=1$ which is a physical relevant jump (see Figure \ref{fig:sfig2} and Figure \ref{fig:fig2} below). The BV solution constructed using the epsilon-neighborhood method satisfies the new energy-dissipation balance but it does not satisfy the energy-dissipation balance.
\item[(v)] The solution constructed by the method in \cite{EfeMie-06} coincides with the solution in (iv).
\end{itemize}

%{\bf Add two more solution: one viscosity correct and one wrong eps-neigh.}
Both solutions in (iii) and (iv) satisfy the definition of BV solutions \cite{MieRosSav-10,MieRosSav-12,MieRosSav-13}. Weak local stability in this case is: $|\partial_x\E(t,x(t))|_* \le 1$.
\end{example}
\begin{figure}
\begin{subfigure}{.5\textwidth}
  \centering
  \includegraphics[width=1\linewidth]{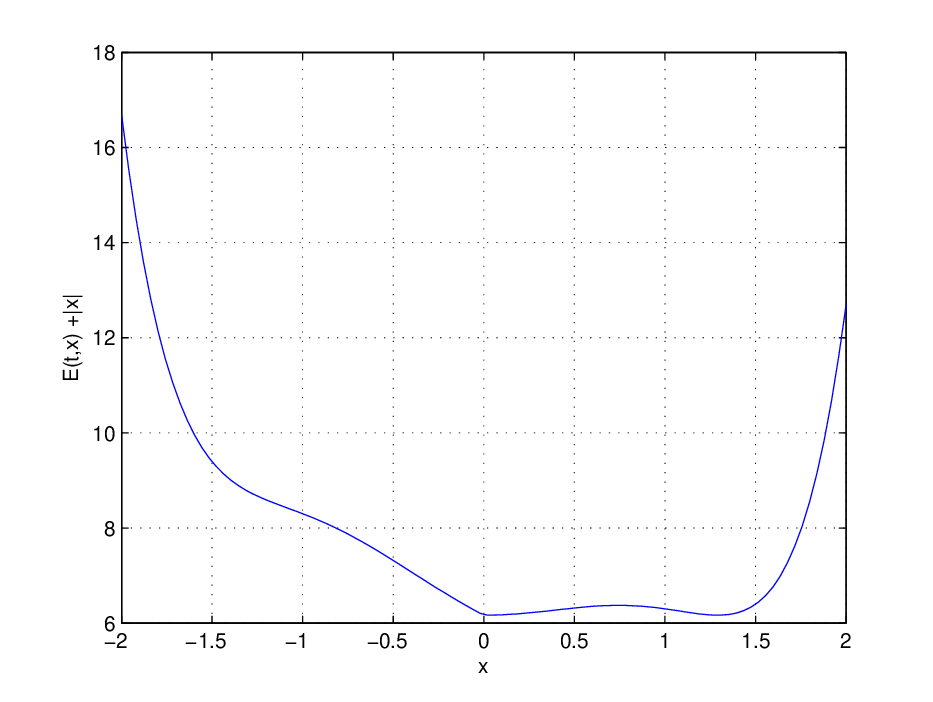}.
  \caption{$\E(t,x)+|x|$ when $t=1/6$}
  \label{fig:sfig1}
\end{subfigure}%
\begin{subfigure}{.5\textwidth}
  \centering
  \includegraphics[width=1\linewidth]{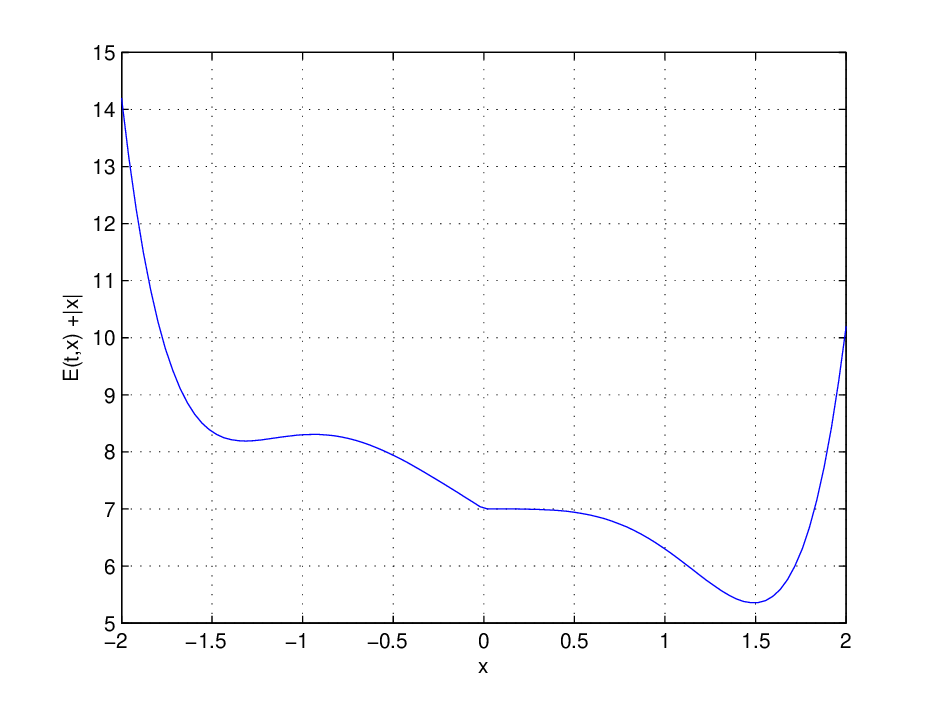}
  \caption{$\E(t,x)+|x|$ when $t=1$}
  \label{fig:sfig2}
\end{subfigure}
\caption{Plots of the total energy at $t=1/6$ and $t=1$ in Example \ref{example2}.}
\label{fig:fig1}
\end{figure}

\begin{figure}
  \centering
  \includegraphics[width=1\linewidth]{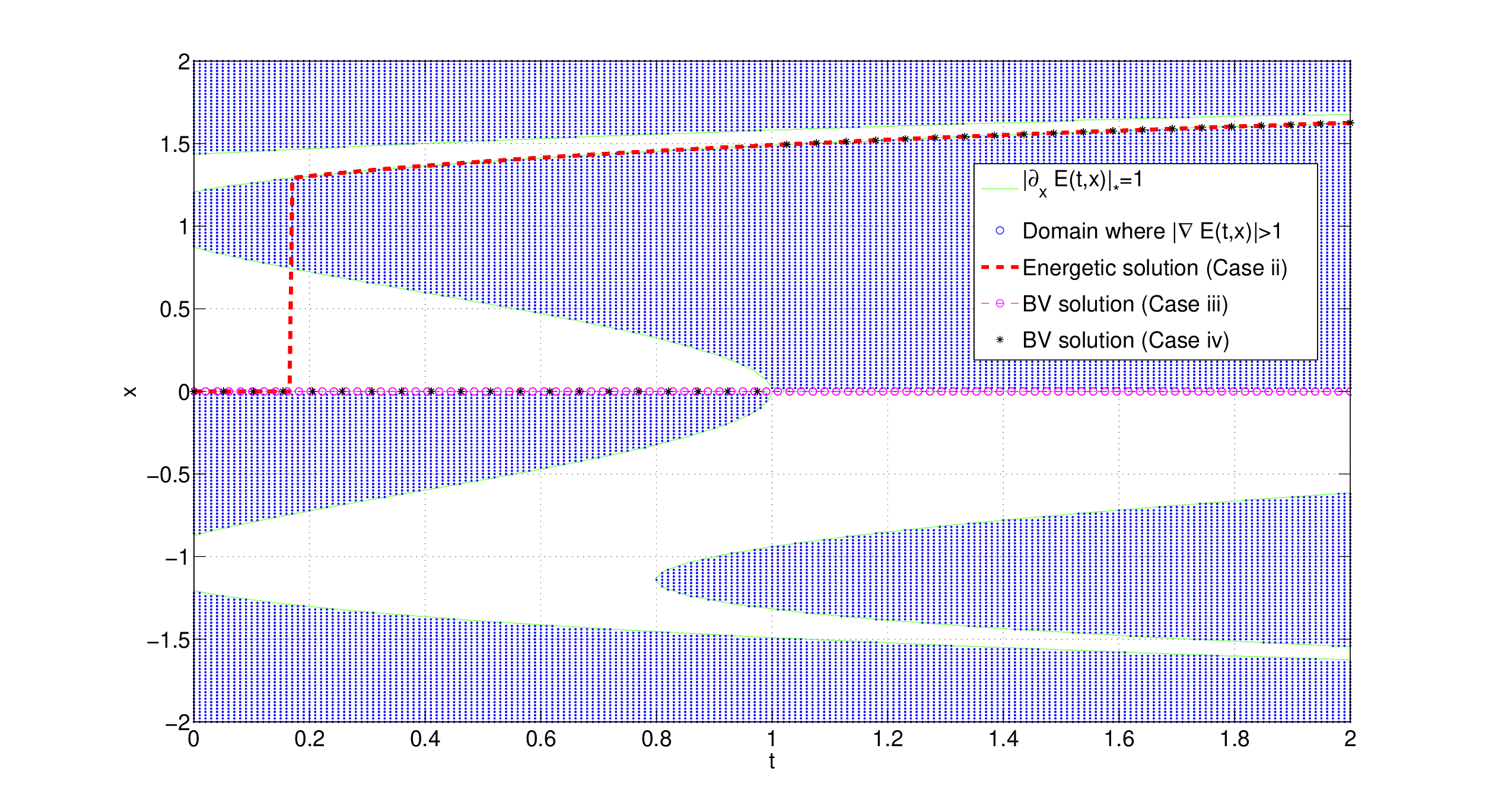}.
\caption{Graph of $|\partial_x \E(t,x)|_* =1$ and three different solutions in Example \ref{example2} are shown. Domain where $|\partial_x \E(t,x)|_*>1$ is filled in.}
\label{fig:fig2}
\end{figure}

%%%%%%%%%%%%%%%%%%%%%%%%%%%%%%%%%%%%%%%%%%%%%%%%%%
%%%%%%%%%%%%%%%%%%%%%%%%%%%%%%%%%%%%%%%%%%%%%%%%%%%%%%%%%%

\section{Epsilon-neighborhood solution $x^\eps$}

We start by considering the discretized solution.

\begin{lemma}[Discretized solution] \label{le:dis-eps-sol}For any given initial state $x_0$, $\eps>0$, $\tau>0$ and partition $0=t_0<t_1<\dots<t_N\leq T$ of $[0,T]$ satisfying $t_n-t_{n-1}=\tau$ and $T \in [\tau N,\tau (N+1))$, there exists a sequence $\{x^{\eps,\tau}_i\}_{i=0}^N$ such that $x_0^{\eps, \tau}=x_0$ and for every $i=1,2,\dots,N$, $x_i^{\eps, \tau}$ minimizes the functional
$$x \mapsto \E(t_i,x)+\Psi(x-x_{i-1}^{\eps,\tau})$$
over $x\in \real^d$ with $\|x-x^{\eps,\tau}_{i-1}\|\le \eps$. 

Moreover, the function $t\mapsto x^{\eps,\tau}(t)$ defined by the interpolation  $x^{\eps,\tau}(t)=x_{i-1}^{\eps,\tau}$ if $t\in [t_{i-1},t_{i})$, $i\in \{1,...,N\}$ satisfies the following energy estimates:
\begin{itemize}
\item[(i)] (Discrete bound) For any $n\in \{1,\dots,N\}$ we have
$$ \E(t_n,x_n^{\eps,\tau})\le \E(0,x_0)\,e^{\lambda t_n}~ {\rm and}~\E(0,x_n^{\eps,\tau})\le \E(0,x_0)\,e^{2\lambda t_n};$$
\item[(ii)] (Integral bound) 
For all \; $ 0\le s \le t \le T$, it holds that 
$\diss_{\Psi} (x^{\eps,\tau};[s,t])<\infty$, $\d_t \E(\cdot,x ^{\eps,\tau}(\cdot))\in L^1(0,T)$ and
$$ \E(t,x^{\eps,\tau}(t))-\E (s,x^{\eps,\tau}(s))\le \int_{s}^t \d_t \E(r, x^{\eps,\tau}(r))\,dr -\diss_{\Psi}(x^{\eps,\tau};[s,t]).$$
\end{itemize}
\end{lemma}
\begin{proof} Since $x \mapsto \E(t_n,x)+\Psi(x-x_{i-1}^{\eps,\tau})$ is continuous, this functional has a minimizer $x_{i}^{\eps,\tau}$ in the compact set $\|x-x^{\eps,\tau}_{i-1}\|\le \eps$. The energy estimates can be proved similarly for energetic solutions (see e.g. \cite{Mielke-06}). A detailed proof can be found in the Appendix.
\end{proof}

\begin{lemma}[Epsilon-neighborhood solution] \label{le:eps-evolution} Given any initial datum $x_0\in \real^d$ such that $x_0$ is a local minimizer of the functional $x\mapsto \E(0,x)+\Psi(x-x_0)$. Let 
$x^{\eps,\tau}$ be as in Lemma \ref{le:dis-eps-sol}. There exists a subsequence 
$\tau_n\to 0$ such that 
$x^{\eps,\tau_n}(t)\to x^{\eps}(t)$ 
for all $t\in [0,T]$. Moreover, the epsilon-neighborhood solution $x^{\eps}(\cdot)$ satisfies the following properties:

\begin{itemize}

\item[(i)] (Epsilon local stability) If $x^{\eps}(\cdot)$ is  
right-continuous at $t$, namely $\lim_{t'\to t^+} x^{\eps}(t')=x^{\eps}(t)$, then $x^{\eps}(t)$ satisfies the epsilon local stability
$$ \E(t,x^{\eps}(t)) \le  \E(t,x)+\Psi(x-x^{\eps}(t))~~{\rm for~all}~\|x-x^{\eps}(t)\|\le \eps;$$

\item[(ii)] (Energy-dissipation inequalities) We have $\diss_{\Psi} (x^{\eps};[0,T])\le C$ (independent of $\eps$), $\d_t \E(\cdot,x ^{\eps}(\cdot))\in L^1(0,T)$ and for all $ 0\le s \le t \le T$,
$$  -\diss_{new}(x^{\eps};[s,t]) \le \E(t,x^{\eps}(t))-\E (s,x^{\eps}(s)) - \int_{s}^t \d_t \E(r, x^{\eps}(r))\,dr\le  -\diss_{\Psi}(x^{\eps};[s,t]).$$
\end{itemize}
\end{lemma}

\begin{proof} {\bf Step 1. Existence.} By the Integral bound in Lemma \ref{le:dis-eps-sol}, the fact that $\E$ is non-negative, and condition (\ref{E2}), we have
\bqq
\diss_{\Psi}(x^{\eps,\tau};[0,T]) &\le& \E(0,x_0) - \E(T,x^{\eps,\tau}(T)) + \int_0^T \d_t \E(r,x^{\eps,\tau}(r))\, dr \hfill\\
&\le& \E(0,x_0) + \sum_{i=1}^{N+1} \int_{t_{i-1}}^{t_i} \lambda \,\E(t_{i-1},x^{\eps,\tau}_{i-1}) \, e^{\lambda(r-t_{i-1})}\, dr.
\eqq
Here, we denote $T$ by $t_{N+1}$. Using the Discrete bound in Lemma \ref{le:dis-eps-sol}, we get
\bqq
\diss_{\Psi}(x^{\eps,\tau};[0,T]) &\le& \E(0,x_0) + \int_0^T \lambda \, \E(0,x_0) \, e^{\lambda r}\, dr= \E(0,x_0)\, e^{\lambda T}.
\eqq
Thus, $\{x^{\eps,\tau}(\cdot)\}$ has uniformly bounded variation and it is uniformly bounded. Therefore, by applying Helly's selection principle \cite{MaiMie-05, AlbDeS-11,Natanson-65}, we can find a subsequence $\tau_n\to 0$ and a BV function $x^{\eps}(\cdot)$ such that $x^{\eps,\tau_n}(t)\to x^{\eps}(t)$ as $n\to \infty$ for all $t\in [0,T]$. 
\text{}\\\\
{\bf Step 2. A consequence of the right-continuity.} Let us denote by $\{t^n_i\}_{i=0}^{N_n}$ the partition corresponding to $\tau_n$ and assume that $t\in [t^n_{i-1},t^n_i)$. It is obvious that
$$x^{\eps,\tau_n}_{i-1}=x^{\eps,\tau_n}(t)\to x^{\eps}(t)$$
as $n\to \infty$. Now we show that if $x^{\eps}(\cdot)$ is right-continuous at $t$, then 
$$x^{\eps,\tau_n}_{i}=x^{\eps,\tau_n}(t_i^n)\to x^{\eps}(t).$$
Let $t'>t$. Thanks to the Integral bound in Lemma \ref{le:dis-eps-sol}, we have
$$ \E(t',x^{\eps,\tau_n}(t'))-\E (t,x^{\eps,\tau_n}(t)) +\diss_{\Psi}(x^{\eps,\tau_n};[t,t']) \le \int_{t}^{t'} \d_t \E(r, x^{\eps,\tau_n}(r))\,dr \le C|t'-t|.$$
Here, the last inequality is due to the continuity of $\d_t\E$ and the fact that $x^{\eps,\tau_n}$ is bounded on $[0,T]$.
For $n$ being large enough, we have $t<t_i^n< t'$. Therefore, 
$$ \Psi(x^{\eps,\tau_n}_i-x^{\eps,\tau_n}_{i-1})\le  \diss_{\Psi}(x^{\eps,\tau_n};[t,t']).$$
Moreover, when $n\to \infty$, we get
$$x^{\eps,\tau_n}(t)\to x^{\eps}(t)~~{\rm and}~x^{\eps,\tau_n}(t')\to x^{\eps}(t').$$
Thus, it follows from the above integral bound that 
$$  \E(t',x^{\eps}(t'))-\E (t,x^{\eps}(t))+ \limsup_{n\to \infty} \Psi(x^{\eps,\tau_n}_i-x^{\eps,\tau_n}_{i-1}) \le C|t'-t|.$$
The inequality above holds for all $t'>t$. Hence, we can take $t'\to t$ and use the assumption $x^{\eps}(t^+)=x^{\eps}(t)$ to obtain
$$\limsup_{n\to \infty} \Psi(x^{\eps,\tau_n}_i-x^{\eps,\tau_n}_{i-1})\le 0.$$
Since $x^{\eps,\tau_n}_{i-1}\to x^{\eps}(t)$, we can conclude that $x^{\eps,\tau_n}_{i}\to x^{\eps}(t)$ as $n \to \infty$.

\text{}\\
{\bf Step 3. Stability.} We show that for all $t\in [0,T]$, if $x^{\eps}(\cdot)$ is right-continuous at $t$, then
$$ \E(t,x^{\eps}(t))\le \E(t,z)+\Psi(z-x^{\eps}(t))~~{\rm for~all}~\|z-x^{\eps}(t)\|\le \eps.$$

To this end, we first prove the result for $z\in \real^d$ with $\|z-x^{\eps}(t)\|<\eps$. Since $\lim_{n\to \infty} x^{\eps,\tau_n}(t) =x^{\eps}(t)$, we get
$$ \|z-x^{\eps,\tau_n}(t)\|<\eps$$
for $n$ large enough. We shall follow the notations in Step 2. The fact that $t \in [t^n_{i-1},t^n_i)$ yields $x^{\eps,\tau_n}(t)=x^{\eps,\tau_n}_{i-1}$. From the definition of $x^{\eps,\tau_n}_{i}$ and condition  $\|z-x^{\eps,\tau_n}_{i-1}\|<\eps$, we obtain
\bqq \label{eq:local-stable-eps-neigh-1}
\E(t^n_i,x^{\eps,\tau_n}_i)+\Psi(x^{\eps,\tau_n}_i-x^{\eps,\tau_n}_{i-1})\le \E(t^n_i,z)+ \Psi(z-x^{\eps,\tau_n}_{i-1}).
\eqq
Taking the limit as $n\to \infty$ and using the fact that both $x^{\eps,\tau_n}_{i-1}$ and $x^{\eps,\tau_n}_i$ converge to $x^{\eps}(t)$ (see Step 2), we have
\bq \label{eq:6Aug-1}
\E(t,x^{\eps}(t)) \le \E(t,z)+ \Psi(z-x^{\eps}(t)) \fall \|z-x^{\eps}(t)\|<\eps. 
\eq

Now for any $z$ satisfying $\|z-x^{\eps}(t)\|=\eps$, we can choose a sequence $z_n$ converging to $z$ such that $\|z_n - x^{\eps}(t)\| < \eps$. Applying (\ref{eq:6Aug-1}) for $z_n$, we get
\bq\label{stability-for-eps-sol-2}
\E(t,x^{\eps}(t)) \le \E(t,z_n) + \Psi(z_n - x^{\eps}(t)).
\eq
Since the mapping $y\mapsto \E(t,y)+ \Psi(y-x^{\eps}(t))$ is continuous, we can take the limit in (\ref{stability-for-eps-sol-2}) and then obtain the result also for $\|z-x^{\eps}(t)\|=\eps$.
\text{}\\\\
{\bf Step 4. Energy-dissipation inequalities.} 

Using the Integral bound in Lemma \ref{le:dis-eps-sol}, we have for all $0\le s\le t\le T$, 
$$ \E(t,x^{\eps,\tau_n}(t))-\E (s,x^{\eps,\tau_n}(s))\le \int_{s}^t \d_t \E(r, x^{\eps,\tau_n}(r))\,dr -\diss_{\Psi}(x^{\eps,\tau_n};[s,t]).$$
Since $x^{\eps,\tau_n}(r)\to x^{\eps}(r)$ for all $r\in [0,T]$, we have
$$ \E(t,x^{\eps,\tau_n}(t)) -\E (s,x^{\eps,\tau_n}(s))\to \E(t,x^{\eps}(t)) -\E (s,x^{\eps}(s))$$
and
$$ \int_{s}^t \d_t \E(r, x^{\eps,\tau_n}(r))\,dr \to  \int_{s}^t \d_t \E(r, x^{\eps}(r))\,dr$$
as $n\to \infty$. Moreover, we have
$$ \liminf_{n\to \infty}\diss_{\Psi}(x^{\eps,\tau_n};[s,t]) \ge \diss_{\Psi}(x^{\eps};[s,t]).$$
Thus we can derive one energy-dissipation inequality
$$ \E(t,x^{\eps}(t))-\E (s,x^{\eps}(s))\le \int_{s}^t \d_t \E(r, x^{\eps}(r))\,dr -\diss_{\Psi}(x^{\eps};[s,t]).$$

We shall use Lemma \ref{2Lemma4} to obtain the other energy-dissipation inequality, 
\begin{eqnarray}\label{lower-1} 
\E(t,x^{\eps}(t))-\E (s,x^{\eps}(s))\ge \int_{s}^t \d_t \E(r, x^{\eps}(r))\,dr -\diss_{ new}(x^{\eps};[s,t]).
\end{eqnarray}
To apply Lemma \ref{2Lemma4}, it is sufficient to verify that $-\nabla_x \E(t,x^{\eps}(t)) \in \partial \Psi(0)$ for a.e. $t\in (0,T)$. Indeed, for every $t\in [0,T]$ such that $x^{\eps}(\cdot)$ is right-continuous at $t$, we have proved in Step 3 the $\eps$-stability 
$$ \E(t,x^{\eps}(t)) \le  \E(t,x)+\Psi(x-x^{\eps}(t))~~{\rm for~all}~\|x-x^{\eps}(t)\|\le \eps.$$
For every $x$ satisfying $\|x-x^{\eps}(t)\|\le \eps$ and for every $s \in [0,1]$, denote by $z=x^{\eps}(t)+s(x-x^{\eps}(t))$. Clearly, $\|z-x^{\eps}(t)\|\le \eps$. Thus,
\[\E(t,x^{\eps}(t)) \le \E(t,z)+\Psi(z-x^{\eps}(t)).\]
This inequality is equivalent to
\[ \frac{\E(t,x^{\eps}(t))-\E(t,x^{\eps}(t)+s(x-x^{\eps}(t)))}{s} \le \Psi(x-x^{\eps}(t)).\]
By taking $s\to 0^+$ and noticing that $\E$ is of class $C^1$, we obtain that 
\[\left<-\nabla_x \E(t,x^{\eps}(t)),x-x^{\eps}(t)\right> \le \Psi(x-x^{\eps}(t)) \; \text{ for all } \; \|x-x^{\eps}(t)\|\le \eps.\]
Now, for every $y \in \mathbb{R}^d \minus \{0\}$, applying the inequality above for $\tilde{y}=x^{\eps}(t)+\eps y/\|y\|$, we get
\[\left<-\nabla_x \E(t,x^{\eps}(t)),y\right> \le \Psi(y).\]
Hence, $-\nabla_x \E(t,x^{\eps}(t)) \in \partial \Psi(0)$ whenever $x^{\eps}(t)$ is right-continuous at $t$.

On the other hand, since $x^{\eps}(\cdot)$ is a BV function, it is continuous except at most countably many points. Thus, we can conclude that $-\nabla_x \E(t,x^{\eps}(t)) \in \partial \Psi(0)$ for a.e. $t \in (0,T)$.
\end{proof}

\begin{lemma}[Lower bound of the new energy-dissipation balance]\label{2Lemma4} For any BV function $u: [0,T] \to \real^d$, energy functional $\E\in C^1([0,T]\times \real^{d})$ and dissipation functional $\Psi$ which is convex and positively $1$-homogeneous, if $-\nabla_x \E(t,u(t))\in \partial \Psi(0)$ for a.e. $t \in (0,T)$, it holds that
$$\E(t_1,u(t_1))-\E(t_0,u(t_0))\geq \int_{t_0}^{t_1}\partial_t \E(s,u(s))\,ds-\diss_{ new}(u;[t_0,t_1]), \; \text{ for all } \; 0\le t_0<t_1\le T.$$
\end{lemma}

This result is due to Mielke, Rossi and Savar\'e (see \cite[Proposition 4.2]{MieRosSav-12} for finite-dimensional space and \cite[Theorem 3.11]{MieRosSav-13} for infinite-dimensional space). For the readers' convenience, a proof of Lemma \ref{2Lemma4} is included in the Appendix.

%%%%%%%%%%%%%%%%%%%%%%%%%%%%%%%%%%%%%%%%%%%%%%%%%%%%%%%%%%%%%%%%%
%%%%%%%%%%%%%%%%%%%%%%%%%%%%%%%%%%%%%%%%%%%%%%%%%%%%%%%%%%%%%%%%%%%%%%
\section{BV solutions constructed using the epsilon-neighborhood method}

\begin{lemma}[Limit of epsilon-neighborhood solutions]\label{le:limit-eps-solution} Given an initial datum $x_0\in \real^d$ which is a local minimizer of the functional $x\mapsto \E(0,x)+\Psi(x-x_0)$. Let $x^{\eps}$ be as in Lemma \ref{le:eps-evolution}. There exists a subsequence 
$\eps_n\to 0$ and a BV function $u$ such that $x^{\eps_n}(t)\to u(t)$ 
for all $t\in [0,T]$. Moreover, the function $u$ satisfies the following properties:

\begin{itemize}

\item[(i)] (Weak local stability) If $t \mapsto u(t)$ is continuous at $t$, then
$$ -\nabla_x \E(t,u(t)) \in \partial \Psi(0);$$

\item[(ii)] (New energy-dissipation balance) For all $0\le s \le t \le T$, one has
$$  \E(t,u(t))-\E (s,u(s))= \int_{s}^t \d_t \E(r, u(r))\,dr-\diss_{new}(u;[s,t]).$$
\end{itemize}
\end{lemma}

\begin{proof} {\bf Step 1. Existence.} Since $\diss_{\Psi}(x^\eps;[0,T])\le C$ is independent of $\eps$, by Helly's selection principle, we can find a subsequence $\eps_n\to 0$ and a BV function $u$ such that $x^{\eps_n}(t)\to u(t)$ as $n\to \infty$ for all $t\in [0,T]$.
\text{}\\\\
{\bf Step 2. Stability.} Let 
$$A:=\{t\in [0,T]\,|\, x^{\eps_n}(\cdot)~{\rm is~ right~ continuous~ at}~t {\rm ~for ~all} ~n\ge 1 \}.$$
Then $[0,T]\minus A$ is at most countable. Moreover, for $t\in A$, by Lemma \ref{le:eps-evolution}, we get
$$ \E(t,x^{\eps_n}(t))\le \E(t,z)+\Psi(z-x^{\eps_n}(t))~~{\rm for ~all}~\|z-x^{\eps_n}(t)\|\le \eps_n$$
for all $n\ge 1$. For $t \in A$ and $n \ge 1$,
$$ \left<-\nabla_x \E(t,x^{\eps_n}(t)), z\right> \le \Psi(z)~~{\rm for~all}~z\in \mathbb{R}^d$$
can be shown in a similar manner as in Step 4, Lemma \ref{le:eps-evolution}.
Taking $n\to \infty$, we obtain
$$ \left<-\nabla_x \E(t,u(t)), z\right> \le \Psi(z)~~{\rm for~all}~z\in \mathbb{R}^d,~{\rm for~all}~t \in A.$$
By continuity, we immediately have $-\nabla_x\E(t,u(t))\in \partial \Psi(0)$ provided that $u$ is continuous at $t$.
\text{}\\\\
{\bf Step 3. New energy-dissipation balance.} By means of a similar proof of the energy inequalities in Lemma \ref{le:eps-evolution}, we have
$$  -\diss_{new}(u;[s,t]) \le \E(t,u(t))-\E (s,u(s)) - \int_{s}^t \d_t \E(r, u(r))\,dr\le  -\diss(u;[s,t]).$$
(The second inequality is a consequence of the corresponding inequality of $x^\eps$ in Lemma \ref{le:eps-evolution} and Fatou's lemma, while the first inequality follows from Lemma \ref{2Lemma4}.)

Note that if the solution $t \mapsto u(t)$ is continuous on $[a,b] \subset [0,T]$, then $\diss(u;[a,b])=\diss_{new}(u;[a,b])$. Thus, we immediately have the energy-dissipation balance
$$\E(b,u(b))-\E (a,u(a)) - \int_{a}^b \d_t \E(r, u(r))\,dr=-\diss(u;[a,b])=-\diss_{new}(u;[a,b]).$$

Therefore, jump points remain to be considered. More precisely, we need to show that if  $u$ jumps at $t\in (0,T)$, namely $u(t^-)\ne u(t^+)$, then
$$\E(t,u(t^+))-\E(t,u(t^-))=-\Delta_{ new}(t,u(t^-),u(t))-\Delta_{ new}(t,u(t),u(t^+)).$$
This fact follows from Lemma \ref{2Lemma4}, \ref{le:discrete-solution} and  \ref{le:upper-bound}.
\end{proof}

To prove the upper bound, we start by showing that the discretized solution $x^{\eps,\tau}$ is ``almost" an optimal transition. 

\begin{lemma}[Approximate optimal transition]\label{le:discrete-solution} 
For the discretized solution $x^{\eps,\tau}$, if we write $x_j:=x^{\eps,\tau}(t_j)$, it holds that
\bqq
\left<-\nabla_x \E(t_i, x_i ), x_i -x_{i-1}\right>= \Psi(x_i-x_{i-1}) + \min_{\eta \in \partial \Psi(0)} \|\eta + \nabla_x\E(t_i,x_i)\|_* \cdot \|x_i-x_{i-1}\|.
\eqq
Consequently, for any $\delta>0$, there exist  $\eps, \tau \le \delta$ and $g(\delta)$ satisfying $g(\delta) \to 0$ as $\delta \to 0$ and 
\bqq
 \E(t,x_{i-1})-\E(t,x_i)\ge \int_{a}^b \Psi(\dot{v}(s)) + \min_{\eta \in \partial \Psi(0)} \|\eta + \nabla_x \E(t,v(s))\|_* \cdot \|\dot{v}(s)\|\, ds - (b-a)g(\delta) \|x_i-x_{i-1}\|,
\eqq
where $t \in [t_{i-1},t_i]$ and $v:[a,b]\to \real^d$ is the linear curve connecting $x_{i-1}$ and $x_i$, namely
$$v(s)=x_{i-1}+\frac{s-a}{b-a} (x_{i}-x_{i-1}).$$
\end{lemma}

\begin{proof} The proof is trivial when $x_i=x_{i-1}$. Hence, we shall assume that $x_i \ne x_{i-1}$.

\text{}\\
{\bf Step 1.} Denote by $m(z):=\|z-x_{i-1}\|$ and $h(z):=\E(t_i,z)+\Psi(z-x_{i-1})$. Recall that $x_i$ is a minimizer for 
$$\inf_{m(z) \le \eps}h(z).$$ 
Denote by $c: =\|x_i-x_{i-1}\|$. Since $c\le \eps$, we can consider $x_{i}$ as a minimizer for 
$$ \inf_{m(z)= c}h(z).$$
Using the Lagrange multiplier, there exists $\lambda\in \real$ such that $ \lambda \nabla m(x_i) \in \partial h(x_i)$, or equivalently
\[\lambda\nabla m(x_i) - \nabla_x \E(t_i,x_i) \in \partial \Psi(x_i-x_{i-1}).\]
%Here, $(x_i-x_{i-1})^T$  stands for the transpose of $(x_i-x_{i-1})$. 
The inclusion above implies two following conditions:
\begin{itemize}
\item[i.] For all $z\in \mathbb{R}^d$, it holds that $\left<\lambda \nabla m(x_i)-\nabla_x \E(t_i,x_i), z \right> \le \Psi(z).$

\item[ii.] $\left<\lambda \nabla m(x_i)-\nabla_x \E(t_i,x_i),x_i-x_{i-1}\right>= \Psi(x_i-x_{i-1}).$
\end{itemize}

\text{}\\
{\bf Step 2.} Since the function $h_1(s)= h(x_{i-1}+s(x_i-x_{i-1}))$ satisfies $h_1(s)\ge h_1(1)$ for all $s\in [0,1]$, it follows that
\[\E(t_i,x_{i-1}+s(x_i-x_{i-1})) + s\Psi(x_i-x_{i-1}) \ge \E(t_i,x_i)+\Psi(x_i-x_{i-1}).\]
The above inequality can be rewritten as
\[\frac{\E(t_i,x_i+(s-1)(x_i-x_{i-1}))-\E(t_i,x_i)}{s-1}+\Psi(x_i-x_{i-1})\le 0.\]
Since $\E$ is of class $C^1$, we can conclude that
\begin{eqnarray}\label{16June}
\left< \nabla_x \E(t_i,x_i),  x_{i} - x_{i-1}\right>+\Psi(x_{i} - x_{i-1}) \le 0.
\end{eqnarray}
In addition, (\ref{16June}) and Condition ii) in Step 1 give $\lambda \le 0$. Moreover, for all $\eta \in \partial \Psi(0)$ we have $-\Psi(x_i-x_{i-1}) \le \left<-\eta , x_i-x_{i-1}\right>$. Thus, condition ii) implies
\begin{eqnarray*}
\left<-\lambda\nabla m(x_i),x_i-x_{i-1}\right> &=& \left<-\nabla_x \E(t_i,x_i),x_i-x_{i-1}\right> - \Psi(x_i-x_{i-1})\hfill\\
&\le&\left<-\nabla_x \E(t_i,x_i)-\eta,x_i-x_{i-1}\right> \hfill\\
&\le& \|-\nabla_x \E(t_i,x_i)-\eta\|_*\cdot \|x_i-x_{i-1}\|.
\end{eqnarray*}
Choosing $\eta_0 = -\nabla_x \E(t_i,x_i)+\lambda \nabla m(x_i)$. Thanks to Condition i) in Step 1, $\eta_0 \in  \partial \Psi(0)$. Moreover, the two inequalities above become equalities with that choice of $\eta_0$. Thus, we can write 
\[\left<-\lambda \nabla m(x_i),x_i-x_{i-1}\right> =\min_{\eta \in \partial \Psi(0)}\|\eta + \nabla_x \E(t_i,x_i)\|_* \cdot \|x_i-x_{i-1}\|.\]
Hence, we obtain that
\bqq
\left<-\nabla_x \E(t_i, x_i ),x_i -x_{i-1}\right>= \Psi(x_i-x_{i-1}) + \min_{\eta \in \partial \Psi(0)} \|\eta + \nabla_x\E(t_i,x_i)\|_* \cdot \|x_i-x_{i-1}\|.
\eqq
\text{}\\
{\bf Step 3.} Consequently, using $|t-t_{i}|\le \delta$, $\|x_{i-1}-x_i\|\le \eps \le \delta$ and the fact that $\nabla_x \E(\cdot, \cdot)$ is continuous on compact sets, there exists $g(\delta)$ such that $g(\delta) \to 0$ when $\delta \to 0$ and
\bqq
\left<-\nabla_x \E(t, v(s)),\dot{v}(s)\right> \ge \Psi(\dot{v}(s)) + \min_{\eta \in \partial \Psi(0)}\|\eta + \nabla_x \E(t,v(s))\|_* \cdot  \|\dot{v}(s)\| - g(\delta) \,\|{\dot{v}(s)}\|
\eqq
for every $s\in [a,b]$. Therefore,
\bqq
 \E(t,x_{i-1})-\E(t,x_i)&=& \int_{a}^b \left<-\nabla_x \E(t, v(s)), \dot{v}(s)\right> \, ds \hfill\\
 &\ge& \int_{a}^b \Psi(\dot{v}(s)) + \min_{\eta \in \partial \Psi(0)} \|\eta + \nabla_x \E(t,v(s))\|_* \cdot \|\dot{v}(s)\|\, ds
  -(b-a) g(\delta) \|x_i-x_{i-1}\|.
\eqq
\end{proof}

Now we are in the position to prove the new energy-dissipation upper bound at jumps.
\begin{lemma}[Upper bound]\label{le:upper-bound} Let $u$ be the function as in Lemma \ref{le:limit-eps-solution}. If $u(t^-) \ne u(t)$, then
\bqq
\Delta_{new}(t,u(t^-),u(t)) \le \E(t, u(t^{-})) -\E(t, u(t )).
\eqq
\end{lemma}
\begin{proof}
Let $0\ll \tau \ll \eps \ll \delta \ll 1$. By the definition of the discretized solution $x^{\eps,\tau}$, for every $t\in (0,T)$ we have
$$ x^{\eps,\tau}(t-\delta)=x^{\eps,\tau}(t_i)\,\,\,{\rm and}\,\, x^{\eps,\tau}(t)=x^{\eps,\tau}(t_{i+k})$$
for $t_i,t_{i+k} \in [t-2\delta,t+\delta]$. 

We can construct an absolutely continuous function $v:[0,1]\to \real^d$ by linearly interpolating the following $(k+3)$ points:
\[\begin{gathered}
  u({t^ - }), {x^{\varepsilon ,\tau }}(t - \delta ) = {x^{\varepsilon ,\tau }}({t_i}) , {x^{\varepsilon ,\tau }}({t_{i + 1}}) , \dots , {x^{\varepsilon ,\tau }}({t_{i+k}}) = {x^{\varepsilon ,\tau }}(t) , u(t). \hfill \\ 
\end{gathered} \]
More precisely, we define   
\[\begin{gathered}
  {z_0} = u({t^ - }), \hfill \\
  {z_1} = x^{\eps,\tau}(t - \delta )= {x^{\varepsilon ,\tau }}({t_i}), \hfill \\
  {z_2} = {x^{\varepsilon ,\tau }}({t_{i + 1}}), \hfill \\
  \dots \hfill \\
  {z_{k + 1}} = {x^{\varepsilon ,\tau }}({t_{i + k}})=x^{\eps,\tau}(t), \hfill \\
  {z_{k + 2}} = u(t), \hfill \\ 
\end{gathered} \]
and denote $r:=1/(k+2)$ and  
$$ v(s)=z_j + \frac{s-jr}{r}(z_{j+1}-z_j)~~{\rm when}~~s\in [jr,(j+1)r],~j=0,1,\dots,k+1.$$
By the definition of the new dissipation, we have
\bqq
\Delta_{new}(t,u(t^-),u(t)) &\le& \int_0^1 \Psi(\dot{v}(s)) + \min_{\eta \in \partial \Psi(0)} \|\eta + \nabla_x \E(t,v(s))\|_* \cdot \|\dot{v}(s)\| \,ds \hfill\\
&=&\sum_{j=0}^{k+1} \int_{jr}^{(j+1)r} \Psi(\dot{v}(s))+ \min_{\eta \in \partial \Psi(0)} \|\eta + \nabla_x \E(t,v(s))\|_* \cdot \|\dot{v}(s)\| \,ds.
\eqq
When $j=0$ and $j=k+1$, we estimate 
\bqq
\int_{jr}^{(j+1)r} \Psi(\dot{v}(s))+ \min_{\eta \in \partial \Psi(0)} \|\eta+\nabla_x \E(t,v(s))\|_* \cdot \|\dot{v}(s)\| \le C\int_{jr}^{(j+1)r} \|\dot{v}(s)\| \, ds=C\|z_{j+1}-z_j\|. 
\eqq 
When $j=1,2,\dots,k$, Lemma \ref{le:discrete-solution} yields the following inequality
\bqq
\int_{jr}^{(j+1)r}\Psi(\dot{v}(s))+ \min_{\eta \in \partial \Psi(0)} \|\eta + \nabla_x \E(t,v(s))\|_* \cdot \|\dot{v}(s)\| \, ds &\le&  \E(t,{x^{\eps,\tau}(t_{i+j - 1})}) - \E(t,{x^{\eps,\tau}(t_{i+j})})\hfill\\
&& + r g(\delta )\cdot \|x^{\eps,\tau}(t_{i+j}) -x^{\eps,\tau}(t_{i+j - 1})\|, 
\eqq 
where $g(\delta)\to 0$ as $\delta\to 0$. Taking the sum over $j=0,1,\dots,k+1$ and using the bound $\Diss_{\Psi} (x^{\eps,\tau};[0,T]) \le C$ (independent of $\eps$ and $\tau$), we find that
\bqq
\Delta_{new}(t,u(t^-),u(t)) &\le & \int_0^1 \Psi(\dot{v}(s)) + \min_{\eta \in \partial \Psi(0)} \|\eta + \nabla_x \E(t,v(s))\|_* \cdot \|\dot{v}(s)\| \,ds \hfill\\
&\le & \E(t, x^{\eps,\tau}(t-\delta ) ) -\E(t, x^{\eps,\tau}(t )) + C g(\delta) + C\|u({t^{-}})-x^{\eps,\tau}(t-\delta)\| + C \|x^{\eps,\tau}(t)-u(t)\|.
\eqq
Taking the limit $\tau\to 0$, then $\eps\to 0$, then $\delta\to 0$, we conclude that
\bqq
\Delta_{new}(t,u(t^-),u(t)) \le \E(t, u(t^{-})) -\E(t, u(t )).
\eqq
%This finishes the proof.
\end{proof}

%%%%%%%%%%%%%%%%%%%%%%%%%%%%%%%%%%%%%%%%%%%%%%%%%%%%%%%%%%%%%%%%%%%%%%%%%%%%
%%%%%%%%%%%%%%%%%%%%%%%%%%%%%%%%%%%%%%%%%%%%%%%%%%%%%%%%%%%%%%%%%%%%%%%%%%%%%%%%
\section{Appendix: Technical proofs}

\subsection{Example \ref{example2}}
First of all, it is easy to verify that $\E(t,x): [0,2] \times \mathbb{R} \to [0,+\infty)$ is $C^1$ and satisfies condition (\ref{eq:E1}). Moreover, $x_0=0$ is a local minimizer for the functional $x \mapsto \E(0,x)+|x|$.
\subsubsection*{\bf Part I. Energetic solution via time-discretization.}
\text{}\\
{\bf Step 1.} Fix a time step $\tau>0$. To find the discretized solution $x^{\tau}(t)$, it is sufficient to calculate $x_i:=x^\tau(t_i)$ where $0=t_0<\dots<t_N\le 1$ and $t_i-t_{i-1}=\tau \fall i=1,2,\dots,N.$ Here $N\in \mathbb{N}$ satisfies $1\in [\tau N, \tau (N+1))$. 

We have $x_0=0$ and for all $i=1,2,\dots,N$, $x_i$ is a minimizer of the functional
$$ x \in \real \mapsto \E(t_i,x)+|x-x_{i-1}|.$$
{\bf Step 2.} Let us fix $t\in (0,2]$ and consider the functional
$$ F(x):=\E(t,x)+|x| = x^2-x^4+0.3\,x^6+t(1-x^2)-x+|x| +6,~~x\in \real.$$ 
It is easy to see that
\begin{itemize}
\item When $t\le 1$, $F(x)$ has two local minimizers (see Figure \ref{fig:sfig1})
$$x=0~~{\rm and}~~x=y(t):=\frac{\sqrt{10+\sqrt{10+90 t}}}{3}.$$
Moreover,
$$F(y(t))-F(0)= \frac{1}{243}(10+\sqrt{10+90t})(8-18t-\sqrt{10+90t}),$$
which is positive if $t<1/6$ and negative if $t>1/6$. Hence $F$ has a unique global minimizer $x=0$ if $0\le t<1/6$, and then $F$ has a unique global minimizer at $x=y(t)$ if $1/6<t<1$.

\item When $t>1$, $F(x)$ has a unique local (also global) minimizer at $x= y(t)$.
\end{itemize}
\text{}\\
{\bf Step 3.} By induction, we can show that if $t_{i_0}<1/6\le t_{i_0+1}$, then $x_i=0$ for all $i=1,2,...,i_0$, and either $x_{i_0+1}=y(t_{i_0+1})$, or $x_{i_0+1}=0$ and $x_{i_0+2}=y(t_{i_0+2})$.  

Next, we show that if $t_{i-1}\ge 1/6$ and $x_{i-1}=y(t_{i-1})>0$, then $x_{i}=y(t_i)$. Recall that $x_i$ is a global minimizer for the functional 
$$x\in \real \mapsto F_{i}(x):=\E(t_{i},x)+|x-x_{i-1}| = x^2-x^4+0.3\,x^6+t_{i}(1-x^2)-x+|x-x_{i-1}|+6.$$
By using the triangle inequality $-x+|x-x_{i-1}|\ge -x_{i-1}$ and the same analysis of $F$, we can conclude that $x_i=y(t_i)$. 

Taking the limit as $\tau \to 0$, we obtain the energetic solution 
\bqq
x(t)=0 ~\text{ if }~ t \in [0,1/6), \; x(1/6)\in \{0, \sqrt{5/3}\}, \; x(t) = y(t) ~\text{ if }  t\in [1/6,2].
\eqq
{\bf Step 4.} Finally, we show that the energetic solution does not satisfy the new energy-dissipation balance. It is sufficient to show that at the jump point $t=1/6$,
\bqq
\E(t,x(t^+)) - \E(t,x(t^-))> - \Delta_{new}(t,x(t^-),x(t^+)).
\eqq
Indeed, a direct computation gives us that at $t=1/6$,
\bqq
\E(t,x(t^+)) - \E(t,x(t^-)) = \E(1/6, \sqrt{5/3}) - \E(1/6, 0) = -\sqrt{5/3}.
\eqq
On the other hand, at $t=1/6$, we have
\bqq
\Delta_{new}(t,x(t^-),x(t^+))
= \int_0^{{\sqrt{15}}/{3}} \max \left \{1,\left|\frac{2}{3}y-4y^3+1.8y^5-1 \right| \right\} dy = \frac{185}{486}+\sqrt{\frac{5}{3}}.
\eqq
Thus, 
$$\E(t,x(t^+)) - \E(t,x(t^-))> - \Delta_{new}(t,x(t^-),x(t^+))~~{\rm at}~t=1/6.$$
\subsubsection*{\bf Part II. BV solution constructed using the viscous dissipation $\Psi_{\eps}(x)=|x|+\eps x^2$.}

We construct the BV solution using the vanishing viscosity method with the viscous term $\eps x^2$ by means of the method used in \cite{MieRosSav-12}. 

Let us briefly recall the construction of the BV solution. Given $\eps>0$ and $\tau>0$. We denote by $e:=\eps/\tau$. Let $0=t_0 < \dots < t_N \le T$ be a partition of $[0,T]$ satisfying $t_i-t_{i-1}=\tau$ for every $i \in \{1, \dots, N\}$ and $T-t_N < \tau$. The discretized problem is to find a sequence $\{x^{\eps,\tau}\}_{i=1}^N$ such that $x^{\eps,\tau}_0=0$ and $x^{\eps,\tau}_{i}$ is a global minimizer for the functional
$$x \in \real \mapsto \{\E(t_i,x)+|x-x_{i-1}^{\tau,\eps}|+e|x-x_{n-1}^{\tau,\eps}|^2\}$$
for every $i=1,2,...,N$ and $e=\eps/\tau.$ Then using interpolation and passing to the pointwise limit as $\tau\to 0,\eps\to 0$ and $e=\eps/\tau\to \infty$, we obtain the BV solution. 

Now coming back to our example, for $t\in (0,2]$, we consider the function
\bqq
F(x):=\E(t,x) + |x| + e|x|^2 = t + (1+e-t)\,x^2-x^4+0.3\,x^6-x+|x|+6,~~x\in \real.
\eqq
If $e$ is large enough (such that $1+e-t\ge 1$), one has
\bqq
F(x) \ge  t + x^2-x^4+0.3\,x^6 +6= t+ \frac{1}{6}x^2 + \left( \sqrt{\frac{5}{6}}x - \sqrt{\frac{3}{10}}x^3\right)^2 +6\ge t+6 = F(0).
\eqq
Thus $F$ has a unique global minimizer at $x=0$. Therefore, the discretized sequence $\{x^{\tau,\eps}_i\}$ is identically equal to $0$ and so is the BV solution. 
\subsubsection* {\bf Part III. BV solution constructed using the epsilon-neighborhood method.}

\text{}\\
{\bf Step 1.} Let $\eps>0$ and $\tau>0$ be small. Let us compute $x_i:=x^{\eps,\tau}(t_i)$, where $t_i=i/N$ for $i=0,1,\dots,N$. Here $N\in \mathbb{N}$ with $1\in [\tau N, \tau (N+1))$. 

By definition, $x_0=0$ and $x_i$ is a minimizer for the functional
$$ F_i(x):=\E(t_i,x)+|x-x_{i-1}|=x^2-x^4+0.3\,x^6+t_i(1-x^2)-x+|x-x_{i-1}|+6$$
over $x\in [x_{i-1}-\eps, x_{i-1}+\eps]$.
In particular, if $x_{i-1}=0$, then $x_i$ is a minimizer for
$$ \tilde{F}_i(x):=x^2-x^4+0.3\,x^6+t_i(1-x^2)-x+|x|+6$$
over $x\in [-\eps, \eps]$. 

Recall that, if $t_i<1$, $\tilde{F}_i(x)$ has two local minimizers at $x=0$ and
$$x=y(t)=\frac{\sqrt{10 + \sqrt{10+90t_i}}}{3}>1.$$
Choose $\eps<1$, then $x=0$ is the unique minimizer for $\tilde{F}_i(x)$ on $x\in [-\eps, \eps]$. Thus, we can conclude that $x_i=0$ whenever $t_i<1$.
\text{}\\\\
{\bf Step 2.} Assume that $t_i\in [1,2]$. We prove that $x_i \le y(t_i)$ for all $i$ by contradiction. Indeed, by means of induction we can assume that $x_{i-1}\le y(t_{i-1})$. Suppose that $x_i > y(t_i)$. Since $x_{i-1}\le y(t_{i-1})<y(t_i)<x_i\le x_{i-1}+\eps$, there exists an $a\in (y(t_i),x_i) \subset [x_{i-1}-\eps,x_{i-1}+\eps]$. Using the fact that the function $x \mapsto g_i(x)=x^2-x^4+0.3x^6+t_i(1-x^2)+6$ is strictly increasing in the interval $[y(t_i),\infty)$ and the triangle inequality $f(x)=-x+|x-x_{i-1}|\ge -x_{i-1}$, we have
\[ F_i(x_i)= x_i^2 -x_i^4 + 0.3x_i^6 + t_i(1-x_i^2)-x_i+|x_i-x_{i-1}| +6>a^2-a^4+0.3a^6+t_i(1-a^2)-x_{i-1} +6 =F_i(a).\]
This contradicts to the assumption that $x_i$ is a minimizer for $F_i(x)$ over $x\in[x_{i-1}-\eps,x_{i-1}+\eps]$. Thus, we must have $x_i \le y(t_i)$.

Moreover, if we choose $\eps <\frac{1}{2}$, it holds that $x_i \ge x_{i-1}$. Indeed, since $g_i(x)$ decreases in $[-\frac{1}{2},y(t_i))$ and $f(x)$ strictly decreases when $x<x_{i-1}$, for every $z \in [-\frac{1}{2},x_{i-1})$
$$F_i(z)=g_i(z)+f(z) > g_i(x_{i-1})+f(x_{i-1})=F_i(x_{i-1}).$$

For the determination of $x_i$, we have the following cases:
\begin{itemize}
\item $x_{i-1}\in[0,y(t_i)-\eps)$. Observe that $y(t)$ strictly increases in $t$. We can choose $\tau$ small enough (in this case $\tau \le \eps$) so that $y(t_{i})-y(t_{i-1})<\eps$.  Thus, $x_{i-1} < y(t_{i-1})$. Since $f(x)=x_{i-1}$ for $x\ge x_{i-1}$ and $g_i(x)$ decreases in the interval $[x_{i-1},y(t_i))$, the function $F_i(x)=g_i(x)+f(x)$ decreases in the interval $[x_{i-1},y(t_i))$. Thus, $x_i =x_{i-1}+\eps$. 

%We show that if $t_i\in [1-2\eps^2,1]$, then $x_{i} \le y(t_i)$. By induction, we can assume that $x_{i-1} \le  y(t_{i-1})$. We assume by contradiction that $x_i>y(t_i)$. 

%Since $x_{i-1}\le y(t_{i-1}) < y(t_i) < x_{i}\le x_{i-1}+\eps$, there exists $a\in (y(t_i),x_i) \cap [x_{i-1}-\eps, x_{i-1}+\eps]$. Then using the fact that the function $x\mapsto x^2-x^4+0.3\,x^6+t_i(1-x^2)+6$ is strictly increasing on $[y(t_i),\infty)$ and the triangle inequality $-x+|x-x_{i-1}|\ge -x_{i-1}$ we have
%\bqq
%F_i(x_i)=x_i^2-x_i^4+0.3x_i^6+t_i(1-x_i^2) -x+|x-x_{i-1}| +6> a^2-a^4+0.3 a^6+t_i(1-a^2) -x_{i-1}+6= F_i(a).
%\eqq 
%This contradicts to the assumption that $x_i$ is a minimizer for $F_i(x)$ over $x\in [x_{i-1}-\eps, x_{i-1}+\eps]$. Thus we must have $x_i\le y(t_i)$. 
\item For the case when $x_{i-1}\in[y(t_i)-\eps,y(t_{i-1})]$, $y(t_i)$ is the unique minimizer of $F_i(x)$ in the interval $[x_{i-1},x_{i-1}+\eps]$. Thus, $x_i=y(t_i)$.
\end{itemize}
%t is straightforward to show that if $x_{i-1}=0$, then $x_i\in \{\pm \eps\}$; and if $x_{i-1}\in (0,y(t_{i-1})]$, then $x_{i}=\min\{x_{i-1}+\eps, y(t_i)\}$.
\text{}\\
{\bf Step 3.} Taking the largest $k$ and the smallest $m$ such that $x_k=0$ and $x_m=y(t_m)$. The number of steps $L$ to move from $x_k$ to $x_m$ is the integer part of $\frac{y(t_m)}{\eps}$. Since $\eps$ is fixed, this value is bounded from above by a constant $C=\left[\frac{5}{\eps}\right]+1$. Hence,
\[t_m = t_k + L\tau \le t_k + C\tau.\]
%this numberTaking the limit as $\tau \to 0$, we obtain that the epsilon-neighborhood solution $x^{\eps}(\cdot)$ satisfies $x^{\eps}(t)=0$ if $t<1$ and $x^{\eps}(t)=y(t)$ for all $t\in [1,2]$ and for all $\eps<1$.

Taking $\tau \to 0$, we have $t_m \approx t_k \approx 1$. Thus, for $\eps <\frac{1}{2}$, the BV solution constructed using the epsilon-neighborhood method is $x(t)=x^{\eps}(t)=0$ if $t \in[0,1)$ and $x(t)=x^{\eps}(t)=y(t)$ if $t\in (1,2]$. At $t=1$, $x(t)$ can take values either $0$ or $y(1)$.
%Taking the limit $\eps\to 0$, we obtain that the BV solution constructed by epsilon-neighborhood method satisfies that $x(t)=0$ if $t\in (0,1)$ and $x(t)=y(t)$ for $t\in (1,2)$.

\text{}\\
{\bf Step 4.} We show that the BV solution constructed by epsilon-neighborhood does not satisfy the energy-dissipation balance. At the jump point $t=1$, one has
$$ -|x(t^-)-x(t^+)|=-\frac{2\sqrt{5}}{3}> \E(t,x(t^+)) - \E(t,x(t^-))=-\frac{400}{243}-\frac{\sqrt{20}}{3}.$$

\subsubsection* {\bf Part IV. The solution constructed by the method in \cite{EfeMie-06}.}
Let us briefly recall the method used in \cite{EfeMie-06}. Let $N\in \mathbb{N}$ be the numbers of time step. The neighborhood is chosen equal to the usual time-step, i.e. $\eps = \tau=\frac{T}{N}$. Take $t_0=0$ and $x_0=0$. For $j=1,\dots,N$, $x_j$ and $t_j$ are defined as follows.
\begin{itemize}
\item $x_j$ minimizes $\{\E(t_{j-1},x)+|x-x_{j-1}|\}$ among all states $x$ such that $|x-x_{j-1}|\le \eps$.
\item $t_j=t_{j-1}+\tau-|x_j-x_{j-1}|$.
\end{itemize}
By the same argument as in Part III, Step 1, we deduce that $x_{i+1}=0$ and $t_{i+1}=t_i+\tau$ if $N<T$ and $t_i<1$.

Now assume that $t_{i}\in [1,2]$. Arguing as in Part III, Step 2, we have $x_{i+1}\in [x_i,y(t_i)]$ and
\begin{itemize}
\item If $x_i \in[0,y(t_i)-\eps)$: $x_{i+1} =x_i+\eps$ and $t_{i+1}=t_i$.
\item If $x_i\in [y(t_i)-\eps,y(t_{i-1})]$: $x_{i+1}=y(t_i)$ and $t_{i+1}=t_i+\tau-|y(t_i)-x_i|$.
\end{itemize}
Taking $\tau$ to $0$, we obtain the solution $x(t)=0$ if $t<1$, $x(t)=y(t)$ if $t\ge 1$.
\text{}\\
\subsection{Proof of the energy estimate in Lemma \ref{le:dis-eps-sol}}
{\bf Step 1.} By the minimality of $x_n^{\eps,\tau}$ at time $t_n$, we have
\bqq \E(t_n,x_n^{\eps,\tau})+\Psi(x_{n}^{\eps,\tau}-x_{n-1}^{\eps,\tau}) \le  \E(t_n,x_{n-1}^{\eps,\tau}) = \E(t_{n-1},x_{n-1}^{\eps,\tau})+\int_{t_{n-1}}^{t_n}\d_t \E(t,x_{n-1}^{\eps,\tau})\,dt.
\eqq
It follows from the assumption (\ref{E2}) that
$$ 
\d_t \E(t,x_{n-1}^{\eps,\tau})\le \lambda \E(t_{n-1},x_{n-1}^{\eps,\tau})\, e^{\lambda(t-t_{n-1})} \fall t\in [t_{n-1},t_n].$$
Applying Gronwall's inequality we obtain
\bqq \E(t_n,x_n^{\eps,\tau})&\le & \E(t_n,x_n^{\eps,\tau})+\Psi(x_{n}^{\eps,\tau}-x_{n-1}^{\eps,\tau}) \hfill\\
& \le & \int_{t_{n-1}}^{t_{n}} \lambda \E(t_{n-1},x_{n-1}^{\eps,\tau})\, e^{\lambda (t-t_{n-1})}dt + \E(t_{n-1},x_{n-1}^{\eps,\tau}) \hfill\\
&=& \E(t_{n-1},x_{n-1}^{\eps,\tau})(e^{\lambda (t_n-t_{n-1})}-1)+\E(t_{n-1},x_{n-1}^{\eps,\tau})=\E(t_{n-1},x_{n-1}^{\eps,\tau})\, e^{\lambda (t_n-t_{n-1})}.
\eqq
By induction,
\bqq \E(t_n,x_n^{\eps,\tau}) &\le & \E(t_{n-1},x_{n-1}^{\eps,\tau}) \, e^{\lambda (t_n-t_{n-1})} \le \E(t_{n-2},x_{n-2}^{\eps,\tau}) \, e^{\lambda (t_{n-1}-t_{n-2})}\, e^{\lambda (t_n-t_{n-1})}\hfill\\
&\le & \dots \le \E(0,x_{0})\, e^{\lambda (t_{1}-t_{0})}\, e^{\lambda (t_{2}-t_{1})} \dots e^{\lambda (t_n-t_{n-1})} = \E(0,x_{0})\, e^{\lambda t_n}.
\eqq
Finally, by (\ref{E2}) again,
$$ \E(0,x_n^{\eps,\tau})\le \E(t_n,x_n^{\eps,\tau})\, e^{\lambda t_n}\le \E(0,x_0)\, e^{2\lambda t_n}.$$
\text{}\\
{\bf Step 2.} Now we prove the integral bound. Assume that $t_{i-1}<s\le t_i <t_{i+1}< \dots <t_j \le t<t_{j+1}$, where $\{t_n\}$ is the partition corresponding to $x^{\eps,\tau}$. We start by writing 
\bq \label{chapter2_eq0}
&~& \E(t,x^{\eps,\tau}(t))-\E(s,x^{\eps,\tau}(s))= \E(t,x^{\eps,\tau}(t))-\E(t_j, x^{\eps,\tau}(t_j))+ \dots  \hfill\\
&~&\,\,\,\,\,\,\,\,\,\,\,\,\,\,\,\,\,\,\,\,\,\,\,\,\,\,\,\,\,\,\,\,\,\,\,+\E(t_j, x^{\eps,\tau}(t_j))-\E(t_{j-1},x^{\eps,\tau}(t_{j-1}))+ \E(t_i, x^{\eps,\tau}(t_i)) -\E(s,x^{\eps,\tau}(s)).\nonumber
\eq
By the minimality of $x_k:=x^{\eps,\tau}(t_k)$ at time $t_k$, we have
\bqq
 \E(t_k,x_k)-\E(t_{k-1},x_{k-1}) &\le & \E(t_{k},x_{k-1})-\Psi(x_{k}-x_{k-1})-\E(t_{k-1},x_{k-1}) = \int_{t_{k-1}}^{t_k} \d_t \E(r, x_{k-1}) \,dr -\Psi(x_{k}-x_{k-1}).
  \eqq
Taking the sum for all $k$ from $i+1$ to $j$ and using $x^{\eps,\tau}(r)=x_{k-1}$ for all $r \in [t_{k-1},t_k)$, we get
 \bq \label{chapter2_eq1}
 \sum_{k=i+1}^j \left[ \E(t_k,x_k)-\E(t_{k-1},x_{k-1})\right] \le \sum_{k=i+1}^j \int_{t_{k-1}}^{t_k} \d_t \E(r,x^{\eps,\tau}(r))\,dr - \sum_{k=i+1}^j \Psi(x_k - x_{k-1}). 
 \eq
Moreover, since $t_{i-1}<s\le t_i$ and $t_j \le t<t_{j+1}$, we can write
\bq 
\E(t_i,x^{\eps,\tau}(t_i))-\E(s,x^{\eps,\tau}(s))&=& \E(t_i,x_i)-\E(s,x_{i-1}) \nonumber \hfill\\
&\le & \E(t_i,x_{i-1})-\Psi(x_{i}-x_{i-1})-\E(s,x_{i-1}) \nonumber \hfill\\
&=& \int_s^{t_i} \d_t \E(r, x^{\eps,\tau}(r))\,dr -\Psi(x_i-x^{\eps,\tau}(s)).\label{chapter2_eq3} \hfill\\
\E(t,x^{\eps,\tau}(t))-\E(t_j,x^{\eps,\tau}(t_j))&=& \E(t,x_j)-\E(t_j,x_{j}) = \int_{t_j}^t \d_t \E(r,x^{\eps,\tau}(r))\,dr-\Psi(x^{\eps,\tau} (t)-x_j) ,\label{chapter2_eq2} 
\eq
Thus, it follows from (\ref{chapter2_eq0}), (\ref{chapter2_eq1}), (\ref{chapter2_eq2}) and (\ref{chapter2_eq3}) that
\bqq
\E(t,x^{\eps,\tau}(t))-\E(s,x^{\eps,\tau}(s)) &\le& \int_s^t \d_t \E(r, x^{\eps,\tau}(r))\,dr- \left(\Psi(x^{\eps,\tau}(t)-x_j) + \sum_{k=i+1}^j \Psi(x_k-x_{k-1}) + \Psi(x_i-x^{\eps,\tau}(s)) \right) \hfill\\
&=& \int_s^t \d_t \E(r, x^{\eps,\tau}(r))\,dr -\diss_{\Psi} (x^{\eps,\tau}; [s,t]).
\eqq

\subsection{Proof of Lemma \ref{2Lemma4}}
\begin{proof}
Applying the chain rule formula for $\E \in C^1$ and $u\in BV$ (see \cite{AmbFusPal-00}), we get
\begin{eqnarray*}
\E(t_1,u(t_1))-\E(t_0,u(t_0)) & = &\int_{t_0}^{t_1}\partial_t\E(s,u(s))\,ds+\int_{t_0}^{t_1} \left<\nabla_x  \E(s,u(s)), u'_{co}(s)\right> \,ds \hfill\\
& ~ &+\sum_{t\in J\cap(t_0,t_1)} \left[ \E(t,u(t))-\E(t,u(t^-))\right] +\sum_{t\in J\cap(t_0,t_1)}\left[ \E(t,u(t^+))-\E(t,u(t)) \right] \hfill\\
& ~ &+\E(t_0,u(t_0^+))-\E(t_0,u(t_0))+\E(t_1,u(t_1))-\E(t_1,u(t_1^-)).
\end{eqnarray*}
The fact that $-\nabla_x \E(t,u(t)) \in \partial \Psi(0)$ whenever $u(t)$ is continuous at $t$ yields
%\[-\nabla_x \E(s,u(s))\cdot u'_{co}(s) \le \Psi(u'_{co}(s)).\]
%Thus,
\begin{eqnarray}\label{M-xinh-dep-1}
\int_{t_0}^{t_1} \left<\nabla_x\E(s,u(s)),u'_{co}(s)\right>\,ds \ge -\int_{t_0}^{t_1} \Psi(u'_{co}(s))\,ds.
\end{eqnarray}
Note that
\begin{eqnarray} \label{2Eq5}
\int_{t_0}^{t_1} \Psi(u'_{co}(s))\,ds &=& \diss (u;[t_0,t_1]) - \sum_{t \in J \cap (t_0,t_1)} \Psi(u(t)-u(t^-)) - \sum_{t \in J \cap (t_0,t_1)} \Psi(u(t^+)-u(t)) \nn \hfill\\
&~& - \Psi(u(t_0^+)-u(t_0))-\Psi(u(t_1)-u(t_1^-)).
\end{eqnarray}
Moreover, for every absolutely continuous curve $v$ in $AC([0,1];\real^d)$ such that $v(0)=u(t^-)$, $v(1)=u(t)$ we have 
\begin{eqnarray*}
|\E(t,u(t))-\E(t,u(t^-))|&=& \left| \int_0^1\left<\nabla_x\E(t,v(s)),\dot{v}(s)\right>\,ds \right|.
\end{eqnarray*}
For any $\eta \in \partial \Psi(0)$, it holds that $\left<\eta , v \right> \le \Psi(v)$ for all $v\in \mathbb{R}^d$. Thus, at every point $s\in [0,1]$ for which the derivative $\dot{v}(s)$ exists, we can write
\[ -\Psi(\dot{v}(s)) \le \left<-\eta , \dot{v}(s)\right>.\]
Hence,
\begin{eqnarray*}
\left<-\nabla_x \E(t,v(s)), \dot{v}(s)\right> &=& \Psi(\dot{v}(s)) - \Psi(\dot{v}(s))+\left< - \nabla_x \E(t,v(s)),\dot{v}(s)\right> \hfill\\
&\le&\Psi(\dot{v}(s)) +\left<- \eta,\dot{v}(s)\right> +\left<-\nabla_x \E(t,v(s)), \dot{v}(s)\right>\hfill\\
&\le& \Psi(\dot{v}(s)) + \|-\eta -\nabla_x \E(t,v(s))\|_* \cdot \|\dot{v}(s)\|.
\end{eqnarray*}
The inequality above holds for every $\eta \in \partial \Psi(0)$. Thus, we obtain
\[\left<-\nabla_x\E(t,v(s)), \dot{v}(s)\right> \le \Psi(\dot{v}(s)) + \inf_{\eta \in \partial \Psi(0)} \|\eta + \nabla_x \E(t,v(s))\|_* \cdot \|\dot{v}(s)\|.\]
Therefore, for any absolutely continuous curve $v$ in $AC([0,1];\mathbb{R}^d)$ satisfying $v(0)=u(t^-), v(1)=u(t)$, it holds that
\begin{eqnarray*}
|\E(t,u(t))-\E(t,u(t^-))| \le \int_0^1 \Psi(\dot{v}(s)) + \inf_{\eta \in \partial \Psi(0)}\|\eta + \nabla_x \E(t,v(s))\|_* \cdot \|\dot{v}(s)\|.
\end{eqnarray*}
By the definition of $\Delta_{new}(t,u(t^-),u(t))$, we can conclude that
\begin{eqnarray} \label{2Eq6}
|\E(t,u(t)) - \E(t,u(t^-))| \le \Delta_{ new}(t,u(t^-),u(t)).
\end{eqnarray}
Similarly, we also get
\bq \label{eq:6Aug-3}
|\E(t,u(t^+))-\E(t,u(t))| \le \Delta_{new} (t,u(t),u(t^+)). 
\eq
Thus, it follows from (\ref{M-xinh-dep-1}), (\ref{2Eq5}), (\ref{2Eq6}) and (\ref{eq:6Aug-3}) that
\begin{eqnarray*}
\E(t_1,u(t_1))-\E(t_0,u(t_0))&\geq&\int_{t_0}^{t_1}\partial_t\E(s,u(s))\,ds-\diss(u;[t_0,t_1]) \hfill\\
& &+\sum_{t\in J\in(t_0,t_1)}\Psi(u(t^-)-u(t)) +\sum_{t\in J\in(t_0,t_1)} \Psi( u(t)-u(t^+)) \hfill\\
& &+\Psi(u(t_0)-u(t_0^+))+\Psi(u(t_1^-)-u(t_1)) \hfill\\
& &-\sum_{t\in J\cap(t_0,t_1)}\Delta_{ new}(t,u(t^-),u(t))-\sum_{t\in J\cap(t_0,t_1)}\Delta_{ new}(t,u(t),u(t^+)) \hfill\\
& &-\Delta_{new}(t_0,u(t_0),u(t_0^+)) - \Delta_{new}(t_1,u(t_1^-),u(t_1)) \hfill\\
&=&\int_{t_0}^{t_1}\partial_t\E(s,u(s))\,ds- \diss_{new}(u;[t_0,t_1]).
\end{eqnarray*}
%This ends the proof of Lemma \ref{2Lemma4}.
\end{proof}
%\text{}\\

%%-----------------------------
%REFERENCES

 \end{document}